\documentclass[a4paper,11pt]{amsart}
\usepackage{amsmath}
\usepackage{mathrsfs}
\usepackage{amsfonts}
\usepackage{graphicx}
\usepackage{color}
\usepackage{amsfonts}
\usepackage{amssymb}
\usepackage[hidelinks]{hyperref}
\usepackage{todonotes}

\usepackage{palatino}

\usepackage[abbrev,backrefs]{amsrefs}

\usepackage{mathtools}

\newtheorem{theorem}{Theorem}[section]

\newtheorem{openquestion}[theorem]{Question}
\newtheorem{corollary}[theorem]{Corollary}

\newtheorem{lemma}[theorem]{Lemma}

\newtheorem{proposition}[theorem]{Proposition}
\newtheorem{remark}[theorem]{Remark}

\numberwithin{equation}{section}

\newcommand{\interior}[1]{%
	{\kern0pt#1}^{\mathrm{o}}%
}

\newcommand{\BV}{\operatorname*{BV}}

\usepackage{scalerel,stackengine}
\stackMath
\newcommand\reallywidehat[1]{%
	\savestack{\tmpbox}{\stretchto{%
			\scaleto{%
				\scalerel*[\widthof{\ensuremath{#1}}]{\kern-.6pt\bigwedge\kern-.6pt}%
				{\rule[-\textheight/2]{1ex}{\textheight}}
			}{\textheight}%
		}{0.5ex}}%
	\stackon[1pt]{#1}{\tmpbox}%
}

\newcommand{\barint}{
	\rule[.036in]{.12in}{.009in}\kern-.16in \displaystyle\int }

\newcommand{\barcal}{\mbox{$ \rule[.036in]{.11in}{.007in}\kern-.128in\int $}}

\usepackage{enumitem}
\makeatletter
\let\@wraptoccontribs\wraptoccontribs
\makeatother

\mathchardef\mhyphen="2D

\title{On the Fourier transform of measures in Besov spaces}
\author[R. Basak]{Riju Basak}
\address[R. Basak]{Department of Mathematics, National Taiwan Normal University, No. 88, Section 4, Tingzhou Road, Wenshan District, Taipei City, Taiwan 116, R.O.C.
}
\email{rijubasak52@ntnu.edu.tw}

\author[D. Spector]{Daniel Spector}
\address[D. Spector]{Department of Mathematics, National Taiwan Normal University, No. 88, Section 4, Tingzhou Road, Wenshan District, Taipei City, Taiwan 116, R.O.C.\\
	\newline
	and
	\newline
	National Center for Theoretical Sciences\\No. 1 Sec. 4 Roosevelt Rd., National Taiwan
	University\\Taipei, 106, Taiwan
	\newline
	and
	\newline
	Department of Mathematics, University of Pittsburgh, Pittsburgh, PA 15261 USA
}
\email{spectda@gapps.ntnu.edu.tw}
\author[D. Stolyarov]{Dmitriy Stolyarov}
\address[D. Stolyarov]{St. Petersburg State University, Department of Mathematics and Computer Science,14th Line 29b, Vasilyevsky Island, St. Petersburg, Russia, 199178}
\email{d.m.stolyarov@spbu.ru}

\begin{document}
	
	\maketitle
	
	\begin{abstract}
		We prove quantitative estimates for the decay of the Fourier transform of the Riesz potential of measures that are in homogeneous Besov spaces of negative exponent:		
		\begin{align*}
			\|\widehat{I_{\alpha}\mu}\|_{L^{p, \infty}} \leq C \|\mu\|_{M_b}^{\frac{1}{2}}\left(\sup_{t>0} t^{\frac{d-\beta}{2}}\|p_{t}\ast \mu\|_{\infty}\right)^{\frac{1}{2}},
		\end{align*}
		where $p=\frac{2d}{2\alpha+\beta}$ with $\beta \in (0,d)$ and $I_\alpha \mu$ is the Riesz potential of $\mu$ of order $\alpha \in ((d-\beta)/2,d-\beta/2)$.  Our results are naturally applicable to the Morrey space $\mathcal{M}^{\beta}$, including for example the Frostman measure $\mu_K$ of any compact set $K$ with $0<\mathcal{H}^\beta(K)<+\infty$ for some $\beta \in (0,d]$.  When $\mu=D\chi_E$ for $\chi_E \in \BV(\mathbb{R}^d)$, $\alpha =1$, and $\beta=d-1$, our results extend the work of Herz and Ko--Lee.  We provide examples which show the sharpness of our results.   
	\end{abstract}

	\section{Introduction}
	The starting point of our research is a result of Ko and Lee \cite[Theorem 1.3]{KL} on the decay of the Fourier transform of the characteristic function of a bounded, measurable set in $\mathbb{R}^d$ whose topological boundary satisfies a certain condition:  
	\begin{theorem}[Ko--Lee]\label{KL}
		Let $\gamma \in (0,d)$ and suppose $E \subset \mathbb{R}^d$ is a bounded, measurable set such that
		\begin{align}\label{boundary}
			\left| \{ x : \operatorname*{dist}(x,\partial E) <\delta\} \right| \lesssim  \delta^{d-\gamma}
		\end{align}
		for all $\delta>0$.  Then
		\begin{align*}
			\widehat{\chi_E} \in L^{\frac{2d}{2d-\gamma},\infty}(\mathbb{R}^d).
		\end{align*}
	\end{theorem}
	\noindent
	Here and in the sequel $\lesssim$ means the left hand side is bounded by a constant depending on the parameters times the right hand side, while for any $1\leq p<\infty$, $L^{p,\infty}(\mathbb{R}^d)$ denotes weak-$ L^{p}(\mathbb{R}^d)$, the Marcinkiewicz space -- the set of measurable $g$ for which the quasi-norm
	\begin{align*}
		\|g\|_{L^{p, \infty}(\mathbb{R}^d)}:=\sup_{t>0} t |\{x \in \mathbb{R}^d : |g(x)|>t\}|^{\frac{1}{p}}
	\end{align*}
	is finite.

	As they observe in \cite[Corollary 1.4]{KL}, it follows that for a bounded set $E\subset \mathbb{R}^d$ with Lipschitz boundary one has
	\begin{align}\label{Lipschitz}
		\widehat{\chi_E} \in L^{\frac{2d}{d+1},\infty}(\mathbb{R}^d),
	\end{align}
	in particular because such sets satisfy \eqref{boundary} with $\gamma=d-1$.  Moreover, an examination of the proof in \cite{KL} shows that \eqref{Lipschitz} comes with a quantitative estimate that for a bounded, measurable set $E \subset \mathbb{R}^d$ with Lipschitz boundary one has
	\begin{align}\label{Lipschitz_quant}
		\|\widehat{\chi_E}\|_{L^{\frac{2d}{d+1},\infty}(\mathbb{R}^d)} \lesssim \mathcal{H}^{d-1}(\partial E)^{1/2},
	\end{align}
	where $\mathcal{H}^{d-1}$ is the Hausdorff measure of dimension $d-1$.  
	
	In view of the development of the theory of sets of finite perimeter, the recording of the inequality \eqref{Lipschitz_quant} immediately suggests that one might try to prove such an inequality for the larger class of sets of finite perimeter, where in place of the Hausdorff measure of the topological boundary one utilizes its perimeter in the sense of De Giorgi.  This idea of a replacement of the assumption \eqref{boundary} related to the topological boundary with a measure theoretic or functional analytic assumption is the main theme of this paper, one which we will shortly see bear fruit in several ways.  The primary objects of interest for us are those elements in the space of finite Radon measures $M_b(\mathbb{R}^d)$ that are also in the homogeneous Besov space of order $\beta-d$ for $\beta \in (0,d]$:
	\begin{align*}
		\dot{B}^{\beta-d}_{\infty,\infty}(\mathbb{R}^d)= \left\{ f \in \mathcal{S}'(\mathbb{R}^d): \|f\|_{\dot{B}^{\beta-d}_{\infty,\infty}} :=  \sup_{t>0} t^{\frac{d-\beta}{2}}\|p_{t}\ast f\|_{\infty}< + \infty\right\}.
	\end{align*}
	Here $p_t$ denotes the standard Euclidean heat kernel
	\begin{align*}
		p_t(x)= \frac{1}{(4\pi t)^{d/2}} e^{-\frac{|x|^2}{4t}} , \, x\in \mathbb{R}^d, \, t>0.
	\end{align*}
	Our choice to introduce the spaces in the parameter $\beta-d$ is related to the numerology of the dimension of these objects.  Here we recall that one has the embedding
	\begin{align}\label{Peetre-embedding}
		\mathcal{M}^\beta(\mathbb{R}^d) \hookrightarrow \dot{B}^{\beta-d}_{\infty,\infty}(\mathbb{R}^d),
	\end{align}
	for $\mathcal{M}^\beta(\mathbb{R}^d)$ the Morrey space given by
	\begin{align*}
		\mathcal{M}^\beta(\mathbb{R}^d) := \left\{ \mu \in \mathcal{M}_{loc}(\mathbb{R}^d) : \sup_{x\in \mathbb{R}^d, r>0} \frac{|\mu|(B(x,r))}{r^\beta}<+\infty\right\},
	\end{align*}
see e.g.~ \cite[Example 7 on p.~165]{PeetreBook}.  Here and in what follows, $B(x,r)$ denotes the open Euclidean ball centered at $x$ of radius $r>0$.  In particular, our work is directly applicable to the elements in the Morrey space of measures of dimension $\beta \in (0,d)$, a point we return to later. 
	
	The first aspect of our viewpoint of these estimates from the perspective of function spaces is a positive result, our 
	\begin{theorem}\label{thm-frac-Y-main}
		Let $\mu$ be finite Radon measure such that $\mu \in \dot{B}^{\beta-d}_{\infty,\infty}(\mathbb{R}^d)$ for $\beta \in (0, d]$. Then for $\alpha \in (0,d)$ such that $\frac{d-\beta}{2}< \alpha < d-\frac{\beta}{2}$, there exist a constant $C=C(d,\alpha, \beta)$ such that 
		\begin{align}\label{eq-reqularity-main}
			\|\widehat{I_{\alpha}\mu}\|_{L^{\frac{2d}{2\alpha+\beta}, \infty}} \leq C \|\mu\|_{M_b}^{\frac{1}{2}} \|\mu\|_{\dot{B}^{\beta-d}_{\infty,\infty}}^{\frac{1}{2}}.
		\end{align}
	\end{theorem}
	\noindent
	Here, $ I_\alpha $ denotes the Riesz potential of order $\alpha \in (0,d)$, whose action on a finite Radon measure $\mu$ is given by
	\begin{align*}
		I_\alpha \mu(x) = \frac{1}{\gamma(\alpha)} \int_{\mathbb{R}^d} \frac{d\mu(y)}{|x - y|^{d - \alpha}},
	\end{align*}
	where $\gamma(\alpha)$ is a normalization constant, see \cite[p.~117]{Stein}.  With the convention
	\begin{align*}
		\widehat{\mu}(\xi) = \int_{\mathbb{R}^d} e^{-2\pi i x \cdot \xi} \;d\mu(x)
	\end{align*}
for the Fourier transform of $\mu \in M_b(\mathbb{R}^d)$, the Fourier transform of $I_\alpha \mu$ is given by 
	\begin{align*}
		\widehat{I_{\alpha}\mu}(\xi)= (2\pi |\xi|)^{-\alpha}\widehat{\mu}(\xi),
	\end{align*}
in the sense of distributions, see \cite[Lemma 2 on p.~117]{Stein}.

\begin{remark}
The range of $\alpha$ admissible in Theorem \ref{thm-frac-Y-main} is related to the exponent $p=\frac{2d}{2\alpha+\beta}$, as this is the range for which $p \in (1,2)$.  This is, in part, related to our method of proof as an interpolation between an $L^1$ and an $L^2$ estimate, and in part to fundamental obstructions for larger values of $p$.  In particular, for $\alpha=0$, $\beta = d-1$, the validity of the estimate would imply that for $f \in L^\infty(S^{d-1})$ one could take $\mu = f\mathcal{H}^{d-1}\llcorner S$ for any $S \subset \mathbb{R}^d$ such that $\mathcal{H}^{d-1}(S)<+\infty$ and obtain
\begin{align*}
\|\widehat{f\mathcal{H}^{d-1}\llcorner S}\|_{L^{\frac{2d}{d-1}, \infty}(\mathbb{R}^d)} \leq C\|f\|_{L^\infty(S)},
\end{align*}
i.e.~the adjoint restriction estimate for $S$.  As no assumptions have been made on the curvature of $S$, such an estimate cannot hold in general, for example when $S$ is the boundary of a cube.
\end{remark}

The framework we adopt of finite Radon measures in negative homogeneous Besov spaces includes sets of finite perimeter:  For $\chi_E \in M_b(\mathbb{R}^d)$, write $D\chi_E$ for the distributional derivative of $\chi_E$.  For a set of finite perimeter one assumes $D\chi_E \in M_b(\mathbb{R}^d)$, i.e. that its total variation $|D\chi_E|(\mathbb{R}^d)$ is finite, so that if we let $\nabla p_t$ to denote the classical gradient of the heat kernel we observe that
	\begin{align}\label{charBV_byparts}
		\|p_t \ast D\chi_E\|_{L^\infty(\mathbb{R}^d)} = \|\nabla p_t \ast \chi_E\|_{L^\infty(\mathbb{R}^d)} \leq ct^{-1/2},
	\end{align}
	i.e. $D\chi_E \in \dot{B}^{-1}_{\infty,\infty}(\mathbb{R}^d)$.  In particular, an application of Theorem \ref{thm-frac-Y-main} with $\beta=d-1$, $\alpha=1$ yields
	\begin{corollary}\label{finite-perimeter}
		Let $E \subset \mathbb{R}^d$ such that $\chi_{E} \in \BV(\mathbb{R}^d)$. Then there exists a positive constant $C=C(d)$ such that 
		\begin{align*}
			\|\widehat{\chi_{E}}\|_{L^{\frac{2d}{d+1}, \infty}} \leq \, C |D\chi_E|(\mathbb{R}^d)^{1/2}.
		\end{align*}
	\end{corollary}
	
	A slight variation of the preceding argument yields a spectrum of results for various orders of differentiability of a set $E$ that one can compare with \cite[Theorem 1.5]{KL}:
	\begin{corollary}\label{Sobolev-function}
		Let $E \subset \mathbb{R}^d$ such that $\chi_{E} \in W^{\eta, p}(\mathbb{R}^d)$ with $0<\eta <1$, $1\leq p <\infty$, and $\eta p <1$, then there exists a positive constant $C=C(\eta,p, d)$ such that
		\begin{align*}
			\|\widehat{\chi_{E}}\|_{L^{\frac{2d}{d+\eta p}, \infty}} \leq C \|\chi_{E}\|_{W^{\eta, p}}^{p/2}.
		\end{align*}
	\end{corollary}	
	\noindent
	Here, $W^{\eta, p}(\mathbb{R}^d)$ is the space defined by the Gagliardo--Slobodeskii semi-norm
	\begin{align*}
		\|\chi_{E}\|^p_{W^{\eta, p}}:= \int_{\mathbb{R}^d} \int_{\mathbb{R}^d} \frac{|\chi_E(x) - \chi_E(y)|^p}{|x-y|^{d+\eta p}}\;dxdy.
	\end{align*}

	A second aspect of our functional analytic perspective is that it gives a unified approach from which to view Ko and Lee's results \cite{KL} alongside the classical work of Herz \cite{Herz} and subsequent contribution of Brandolini, Colzani, and Travaglini \cite{BCT} that they build upon.  Here it is useful to break down these several results into two implications.  The first of these implications is the translation of assumptions on $E$ into its inclusion in a Besov space:
	\begin{align}\label{unifying_idea}
		\text{Assumptions on } E \implies \chi_E \in \dot{B}^{\frac{d-\gamma}{2}}_{2,\infty}(\mathbb{R}^d)
	\end{align}
	for some $\gamma \in (0,d)$.  We recall that one equivalent semi-norm on the homogeneous Besov space $\dot{B}^{\frac{d-\gamma}{2}}_{2,\infty}(\mathbb{R}^d)$, which is easy to relate with our results is given by
	\begin{align*}
		\|a\|_{\dot{B}^{\frac{d-\gamma}{2}}_{2,\infty}} := \sup_{t>0} t^{(\gamma-d)/4} \|p_{2t} \ast a - p_t\ast a\|_{L^2},
	\end{align*}
	though a more modern approach to harmonic analysis and Besov spaces would replace $(p_{2t}-p_t)$ with a compactly supported Littlewood-Paley projector (see \cite[p. 93]{Triebel}).  
	
	In the work of Herz, $\gamma =d-1$, and \eqref{unifying_idea} is a consequence of \cite[Theorem 2]{Herz} (see also \cite{BG} for a similar estimate and asymptotics for $u \in \BV(\mathbb{R}^d) \cap L^\infty(\mathbb{R}^d)$):
	\begin{align}\label{L2_decay_boundary}
		\int_{B(0,R)} |\xi|^2 |\widehat{\chi_E}(\xi)|^2\;d\xi \lesssim R\;\mathcal{H}^{d-1}(\partial E).
	\end{align}
Brandolini, Colzani, and Travaglini also treat the case $\gamma =d-1$, while Ko and Lee establish such an embedding for any $\gamma \in (0,d)$, where the corresponding translation of regularity is \cite[Lemma 2.10]{BCT}, \cite[Lemma 2.4]{KL}:
	\begin{align*}
		\| \chi_E\|_{ \dot{B}^{\frac{d-\gamma}{2}}_{2,\infty}} \lesssim
		\sup_{\delta>0} \frac{\left| \{ x : \operatorname*{dist}(x,\partial E) <\delta\} \right|}{  \delta^{d-\gamma}}.
	\end{align*}
	This analysis reveals that the assumption \eqref{boundary} is only meaningful for $\gamma \in [d-1,d)$, as the argument of Sickel \cite[p.~401]{Sickel} shows that for $E$ a measurable set, $0<|E|<+\infty$ is in contradiction with $\chi_E \in \dot{B}^{\eta}_{2,\infty}(\mathbb{R}^d)$ for any $\eta>1/2$.  
	
	This highlights a third useful aspect of our framework, that in place of estimates for the Fourier transform of the characteristic function of measurable sets with certain regularity assumptions, one can utilize measures associated to sets.  For example, given $\beta \in (0,d)$ and a compact set $K\subset \mathbb{R}^d$ with $0<\mathcal{H}^\beta(K)<+\infty$, Frostman's lemma \cite[p. 18]{Mattila} yields a probability measure $\mu_K$ for which
	\begin{align*}
		\|\mu_K\|_{\mathcal{M}^\beta}:=\sup_{x\in \mathbb{R}^d, r>0} \frac{\mu_K(B(x,r))}{r^\beta}<+\infty.
	\end{align*}

	Then the embedding noted in \eqref{Peetre-embedding} yields that such a measure is admissible in our framework and thus from our results one deduces an estimate for the decay of the Fourier transform of $I_\alpha \mu_K$ for any such Frostman measure on a set of dimension $\beta$, for a suitable range of $\alpha>0$.  
	
	The second of the implications of the papers \cite{Herz, KL} is the translation of Besov regularity of a function into the decay of its Fourier transform.  The sharp result here is Ko and Lee's establishment of the boundedness of the map (see \cite[Proof of Theorem 1.3 on p.~1104]{KL})
	\begin{align}\label{KoLeeBesov-to-WeakLp}
		\mathcal{F} : \dot{B}_{2,\infty}^{\frac{d-\gamma}{2}}(\mathbb{R}^d) &\to L^{\frac{2d}{2d-\gamma},\infty}(\mathbb{R}^d).
	\end{align}
	Their work extends Herz's work in two directions, firstly in the parameter $\gamma \in [d-1,d)$ and secondly that decay on annuli not only leads to an inclusion in $L^p$ for $p>\frac{2d}{2d-\gamma}$ (which is remarked by Herz just after Theorem 2, see \cite[p.~82]{Herz}), but that one has a weak-type estimate at the endpoint.  Our contribution is the observation that it is useful to mediate between a set $E$ and its derivative via the introduction of the Riesz potential, which is easily incorporated into our analysis via the known Riesz potential embedding
	\begin{align*}
		I_\alpha :  \dot{B}_{2,\infty}^{\frac{\beta-d}{2}}(\mathbb{R}^d) &\to  \dot{B}_{2,\infty}^{\frac{\beta-d}{2}+\alpha}(\mathbb{R}^d).
	\end{align*}
	This integrates easily with Ko and Lee's embedding \eqref{KoLeeBesov-to-WeakLp} to obtain our results, where one replaces $\dot{B}_{2,\infty}^{\frac{d-\gamma}{2}}(\mathbb{R}^d)$ in \eqref{KoLeeBesov-to-WeakLp} with $\dot{B}_{2,\infty}^{\frac{\beta-d}{2}+\alpha}(\mathbb{R}^d)$.  This explains the restriction on $\alpha$, as one requires $\frac{\beta-d}{2}+\alpha>0$ for the validity of an analogue of \eqref{KoLeeBesov-to-WeakLp}.

	A fourth aspect of our work relates to sharpness.  In particular, in the work of Ko and Lee \cite[p.~1099]{KL} they note that for $\gamma=d-1$ the exponent $\frac{2d}{d+1}$ is optimal, and while it seems to be true for other values they are not able to construct an example.  We establish here the sharpness of our results, as a consequence of the following theorems.  We first prove that for every $\beta \in (0,d]$, the exponent $p=\frac{2d}{2\alpha+\beta}$ cannot be improved in
	\begin{theorem}\label{sharp_p}
		Let $\beta \in (0, d]$ and $\alpha \in (0,d)$ be such that $\frac{d-\beta}{2}< \alpha < d-\frac{\beta}{2}$.  Then for any $r<\frac{2d}{2\alpha+\beta}$ there exists a sequence of functions $\Phi_N \in M_b \cap \dot{B}^{\beta-d}_{\infty,\infty}(\mathbb{R}^d)$ such that
		\begin{align*}
			\|\Phi_N\|_{M_b}, \|\Phi_N\|_{\dot{B}^{\beta-d}_{\infty,\infty}} \lesssim 1
		\end{align*}
		and
		\begin{align*}
			\|\widehat{I_{\alpha}\Phi_N}\|_{L^{r, \infty}} \to \infty.
		\end{align*}
	\end{theorem}
	
	Next, we have the following theorem which shows that when $\beta =k \in \mathbb{N}\cap (0, d]$, not only does one have sharpness of the exponent  $\frac{2d}{2\alpha+k}$, even the second parameter in the Lorentz space cannot be improved.  
	\begin{theorem}\label{sharp_q}
		Let $k \in \mathbb{N}\cap (0, d-1]$ and $\alpha \in (0,d)$ be such that $\frac{d-k}{2}< \alpha < d-\frac{k}{2}$.  The measure $\mu_k= \mathcal{H}^{k}\llcorner{S^{k}} \in M_b \cap \dot{B}^{k-d}_{\infty,\infty}$ is such that
		\begin{align*}
			\|\mu_k\|_{L^{\frac{2d}{2\alpha+k}, q}} = +\infty.		\end{align*}
		for all $0< q<\infty$.
	\end{theorem}
\noindent
Here, $S^{k} \subset \mathbb{R}^{k+1}$ is a $k$-dimensional unit sphere for any decomposition $\mathbb{R}^d=\mathbb{R}^{k+1} \times \mathbb{R}^{d-(k+1)}$ and $\mathcal{H}^{k}\llcorner S^{k}$ denotes the $k$-dimensional Hausdorff measure on $S^k$.
	
	The conclusion of Theorem \ref{sharp_q} is stronger than that in Theorem \ref{sharp_p} in two respects: first, we show that one is able to find a single measure for which the embedding fails;  second, for this measure we show this failure at the critical value for all finite values of the second Lorentz parameter.  The complication in the former counterexample is because we are not in the integer regime, where a direct computation of the decay of the Fourier transform of lower dimensional spheres yields the result.  Instead, we argue via a random construction and Khintchine's inequality.  This discrepancy prompts
	\begin{openquestion}\label{q1}
		Given $\beta \in (0, d]$, $\alpha \in (0,d)$ such that $\frac{d-\beta}{2}< \alpha < d-\frac{\beta}{2}$, and $1\leq q<+\infty$, can one find a sequence of functions $\Phi_N \in M_b \cap \dot{B}^{\beta-d}_{\infty,\infty}(\mathbb{R}^d)$ such that
		\begin{align*}
			\|\Phi_N\|_{M_b}, \|\Phi_N\|_{\dot{B}^{\beta-d}_{\infty,\infty}} \lesssim 1
		\end{align*}
		and
		\begin{align*}
			\|\widehat{I_{\alpha}\Phi_N}\|_{L^{\frac{2d}{2\alpha+\beta}, q}} \to \infty?
		\end{align*}
	\end{openquestion}
	
	A more challenging problem is whether the counterexample can be constructed using probability measures associated to sets in $\mathbb{R}^d$, or optimally a single probability measure.  One poses
	\begin{openquestion}\label{q2}
		In Theorem \ref{sharp_p} and Question \ref{q1}, can one take $\Phi_N = \mu_N$, for $\mu_N$ a probability measure on some set $E_N \subset \mathbb{R}^d$?
	\end{openquestion}
	
 Question \ref{q2} is about a measure that does not concentrate too much and whose Fourier transform is large in a sense. It intersects a problem of classical interest, the construction of sets with prescribed Fourier dimension. We broadly refer to Section~$3.6$ in~\cite{Mattila} and now give a brief outline. 
	
	We recall the Fourier dimension of a set $E$ is the supremum of all $\eta \in [0,d]$ such that
	\begin{align}\label{Fourier_dimension}
		|\widehat{\mu_E}(\xi)| \leq C |\xi|^{-\eta/2}
	\end{align} 
	for some non-trivial Borel measure for which $\operatorname*{supp} \mu_E \subset E$. The Fourier dimension is always smaller than or equal to the Hausdorff dimension of a set. A set is called a Salem set if the two dimensions coincide. The problem of construction of Salem sets is highly non-trivial, as the classical constructions of Salem \cite{Salem}, Kahane \cite{Kahane}, and Bluhm \cite{Bluhm} are probabilistic.  The more recent development of explicit constructions rely heavily on algebraic number theory, see Kaufman \cite{Kaufman}, Hambrook \cite{Hambrook,Hambrook1}, Fraser and Hambrook \cite{FraserHambrook}.  In particular, one observes that the question of Salem sets concerns constructing a measure in a Morrey space with small Fourier transform, small being quantified by \eqref{Fourier_dimension}.  By constrast, we seek the construction of a measure in a Morrey space with large Fourier transform, where large might take the form of a measure-theoretic analogue of the reverse inequality to \eqref{Fourier_dimension}, i.e.
	\begin{align*}
		| \{& \xi \in B(0,R) \setminus B(0,R/2) : |\widehat{\mu_E}(\xi)| >t  \}| \\
		&\quad \geq c_1 | \{ \xi \in B(0,R) \setminus B(0,R/2) : |\xi|^{-\eta/2} >c_2t  \}|, \quad t \in (0,1),
	\end{align*}
	for some $c_1,c_2>0$ and all $R>0$ sufficiently large.  One approach to address our Question \ref{q2} would be to consider the sharpness in these various constructions of Salem sets, though with an additional flexibility, as our problem does not require the upper bound that makes the constructions so technical.

	\section{Preliminaries}\label{preliminaries}
	In this section we recall several basic results that will be useful in the sequel.  	
	First, we have some standard estimates for the heat kernel $p_{t}$:
	\begin{align}\label{heat-kernel-est}
		\int_{\mathbb{R}^d} p_{t}(x)\, dx=1 \quad \text{and} \quad \int_{\mathbb{R}^d} |\nabla^k p_{t}(x)| \, dx \leq C_{d,k} t^{-k/2},
	\end{align}
	for some positive constant $ C_{d,k}$.  We also utilize a slight variant of the second estimate, that for $\eta \in (0,d)$ one has
	\begin{align}\label{heat-kernel-est-2}
		\int_{\mathbb{R}^d} |(-\Delta)^{\eta/2} p_{t}(x)| \, dx \leq C_{d,\eta} t^{-\eta/2},
	\end{align}
	for some positive constant $ C_{d,\eta}$.  These two estimates follow easily from scaling, once one verifies the summability of derivatives of the Gaussian.
	
	The following lemma gives examples of functions for which our results can be applied.
	\begin{lemma}\label{lem-higher-grad-measure}
		If $u \in L^{\infty}(\mathbb{R}^d)$, then for all $k \in \mathbb{N}\cap [0,d]$, $D(\nabla^{k-1}u) \in \dot{B}^{-k}_{\infty,\infty}(\mathbb{R}^d)$ and
		\begin{align*}
			\|D(\nabla^{k-1}u)\|_{\dot{B}^{-k}_{\infty,\infty}} \lesssim \|u\|_{\infty}.
		\end{align*}
	\end{lemma}
	\begin{proof}
		From the definition of the distributional derivatives $D(\nabla^{k-1}u)$ and the estimate \eqref{heat-kernel-est} we have
		\begin{align*}
			\|p_t \ast D(\nabla^{k-1}u)\|_{\infty} =  \|\nabla^k p_t \ast u\|_{\infty} \leq \|\nabla^k p_t\|_{1} \|u\|_{\infty} \lesssim &  t^{-k/2} \|u\|_{\infty},
		\end{align*}
		which implies the claim.		
	\end{proof}
	
	%
	
	One can also replace the integer derivatives with a fractional Laplacian:
	
	\begin{lemma}\label{lem-higher-fraclaplace-measure}
		If $u \in L^{\infty}(\mathbb{R}^d)$, then for all $\eta\in (0,d)$,  $(-\Delta)^{\eta/2}  u \in \dot{B}^{-\eta}_{\infty,\infty}$ and
		\begin{align*}
			\|(-\Delta)^{\eta/2}  u \|_{\dot{B}^{-\eta}_{\infty,\infty}} \lesssim \|u\|_{\infty}.
		\end{align*}
	\end{lemma}
	\begin{proof}
		From the distributional defintion of $(-\Delta)^{\eta/2}  u$ and \eqref{heat-kernel-est-2} we have
		\begin{align*}
			\|p_t \ast (-\Delta)^{\eta/2}  u\|_{\infty} = & \|(-\Delta)^{\eta/2}  p_t \ast u\|_{\infty}\\
			\leq & \|(-\Delta)^{\eta/2}  p_t\|_{1} \|u\|_{\infty}\\
			\lesssim &  t^{-\eta/2} \|u\|_{\infty},
		\end{align*}
		which implies the claim.		
	\end{proof}

	\section{Proof of Main results}\label{proofs}
	
	\begin{proposition}\label{Lq-estimate}
		Let $\mu$ be a finite Radon measure such that $\mu \in \dot{B}^{\beta-d}_{\infty,\infty}(\mathbb{R}^d)$. Then for $1\leq q \leq \infty$ and for all $t>0$, we have
		\begin{align*}
			t^{\frac{(d-\beta)}{2}(1-\frac{1}{q})} \|p_t\ast \mu \|_q \leq \|\mu\|_{M_b}^{\frac{1}{q}} \|\mu\|_{\dot{B}^{\beta-d}_{\infty,\infty}}^{1-\frac{1}{q}}.
		\end{align*}
	\end{proposition}
	\begin{proof}
		As $\mu \in M_b(\mathbb{R}^d)$, by Fubini's theorem we have 
		\begin{align}\label{L1-esti}
			\|p_t\ast \mu \|_1 \leq \|p_t\|_1 \|\mu\|_{M_b} \leq \|\mu\|_{M_b}.
		\end{align}
		Since $\mu \in \dot{B}^{\beta-d}_{\infty,\infty}$, by definition we have 
		\begin{align}\label{L-infty-esti}
			\|p_t\ast \mu \|_{\infty} \leq   t^{\frac{\beta-d}{2}} \|\mu\|_{\dot{B}^{\beta-d}_{\infty,\infty}}
		\end{align}
		for all $t>0$.  The combination of the estimates \eqref{L1-esti} and \eqref{L-infty-esti} and H\"older's inequality yields, for $1<q<\infty$, the inequality
		\begin{align*}
			t^{\frac{(d-\beta)}{2}(1-\frac{1}{q})} \|p_t\ast \mu \|_q \leq \|\mu\|_{M_b}^{\frac{1}{q}} \|\mu\|_{\dot{B}^{\beta-d}_{\infty,\infty}}^{1-\frac{1}{q}} ,
		\end{align*}
		for all $t>0$.  This completes the proof.
	\end{proof}	
	
	\begin{lemma}\label{lem-L2-average}
		Let $\mu$ be a finite Radon measure such that $\mu \in \dot{B}^{\beta-d}_{\infty,\infty}$ for $\beta \in \mathbb{R}$. Then there exist a constant $C=C(d)$ such that 
		\begin{align}\label{L2-average}
			\frac{1}{R^{d-\beta}} \int_{B(0, R)} |\widehat{\mu}(\xi)|^2 \, d\xi \leq C \|\mu\|_{M_b} \|\mu\|_{\dot{B}^{\beta-d}_{\infty,\infty}}
		\end{align} 
		for all $R>0$.
	\end{lemma}
	\begin{proof}
		First note that for $q=2$, Proposition \ref{Lq-estimate} gives
		\begin{align*}
			t^{\frac{d-\beta}{4}} \|p_t \ast \mu\|_2 \leq \|\mu\|_{M_b}^{\frac{1}{2}} \|\mu\|_{\dot{B}^{\beta-d}_{\infty,\infty}}^{\frac{1}{2}}
		\end{align*}
		for all $t>0$.  The preceding estimate and  Plancherel formula's imply
		\begin{align*}
			t^{\frac{d-\beta}{2}} \int_{\mathbb{R}^d} e^{-8\pi^2 t|\xi|^2} |\widehat{\mu}(\xi)|^2 \, d\xi \leq \|\mu\|_{M_b} \|\mu\|_{\dot{B}^{\beta-d}_{\infty,\infty}}.
		\end{align*}
		We set $t=\frac{1}{R^2}$ in the above inequality to obtain
		\begin{align*}
			\frac{1}{R^{d-\beta}} \int_{\mathbb{R}^d} e^{-8\pi^2 \frac{|\xi|^2}{R^2}} |\widehat{\mu}(\xi)|^2 \, d\xi \leq  \|\mu\|_{M_b} \|\mu\|_{\dot{B}^{\beta-d}_{\infty,\infty}}.
		\end{align*}
		Since $e^{-8\pi^2} \leq e^{-8\pi^2 \frac{|\xi|^2}{R^2}}$ for $\xi \in B(0, R)$, we have
		\begin{align*}
			\frac{1}{R^{d-\beta}} \int_{B(0, R)} |\widehat{\mu}(\xi)|^2 \, d\xi
			&\leq  \frac{e^{8\pi^2}}{R^{d-\beta}}  \int_{B(0, R)} e^{-8\pi^2\frac{|\xi|^2}{R^2}} |\widehat{\mu}(\xi)|^2 \, d\xi \\
			&\leq   \,  \frac{e^{8\pi^2}}{R^{d-\beta}} \int_{\mathbb{R}^d} e^{-8\pi^2\frac{|\xi|^2}{R^2}} |\widehat{\mu}(\xi)|^2 \, d\xi \\
			&\leq  \, C \|\mu\|_{M_b} \|\mu\|_{\dot{B}^{\beta-d}_{\infty,\infty}}
		\end{align*}
		for all $R>0$.  This completes the proof the lemma.
	\end{proof}

	\begin{proof}[Proof of Theorem \ref{thm-frac-Y-main}]
		For any $N_{0} \in \mathbb{Z}$ fixed to be chosen later we have 
		\begin{align*}
			\left| \left\{ \xi \in \mathbb{R}^d : |\widehat{I_{\alpha}\mu}(\xi)| > \lambda \right\}\right| &\leq  \sum_{k=-\infty}^{N_0}  \left| \left\{ \xi \in B(0, 2^{k+1})\setminus B(0, 2^{k}) : |\widehat{I_{\alpha}\mu}(\xi)| > \lambda \right\}\right| \\
			&\quad+ \sum_{k=N_0+1}^{\infty}  \left| \left\{ \xi \in B(0, 2^{k+1})\setminus B(0, 2^{k}) : |\widehat{I_{\alpha}\mu}(\xi)| > \lambda \right\}\right| \\
			&:= I_1 +  I_2 .
		\end{align*}
		
		We first estimate $I_2$. By Chebyshev's inequality, we have 
		\begin{align*}
			I_2 &\leq  \frac{1}{\lambda^2} \sum_{k=N_0+1}^{\infty} \int_{B(0, 2^{k+1})\setminus B(0, 2^k)} |\widehat{I_{\alpha}\mu}(\xi)|^2 \, d\xi \\
			&=  (2\pi)^{-2\alpha}\frac{1}{\lambda^2} \sum_{k=N_0+1}^{\infty} \int_{B(0, 2^{k+1})\setminus B(0, 2^k)} |\xi|^{-2\alpha} |\widehat{\mu}(\xi)|^2 \, d\xi \\ 
			&\leq  (2\pi)^{-2\alpha} \frac{1}{\lambda^2} \sum_{k=N_0+1}^{\infty} 2^{-2k \alpha} \int_{B(0, 2^{k+1})} |\widehat{\mu}(\xi)|^2 \, d\xi.
		\end{align*}
		The preceding chain of inequalities in combination with inequality \eqref{L2-average} from Lemma \ref{lem-L2-average} yields 
		\begin{align*}
			I_2  &\leq C \frac{1}{\lambda^2} \sum_{k=N_0+1}^{\infty} 2^{-2k \alpha} 2^{k(d-\beta)}  \|\mu\|_{M_b} \|\mu\|_{\dot{B}^{\beta-d}_{\infty,\infty}}\\
			&\lesssim  \frac{1}{\lambda^2} 2^{-N_0(2\alpha-(d-\beta))}  \|\mu\|_{M_b} \|\mu\|_{\dot{B}^{\beta-d}_{\infty,\infty}},	
		\end{align*} 
		where in the last inequality we utilize the assumption that $\alpha > \frac{d-\beta}{2}$.
		
		To estimate $I_1$, again using Chebyshev's inequality, we write
		\begin{align*}
			I_1 \leq \frac{1}{\lambda} \sum_{k=-\infty}^{N_0} \int_{B(0, 2^{k+1})\setminus B(0, 2^k)} |\widehat{I_{\alpha}\mu}(\xi)| \, d\xi.
		\end{align*}
		This inequality, in combination with Holder's inequality and inequality \eqref{L2-average} from Lemma \ref{lem-L2-average}, yields
		\begin{align*}
			I_1&\lesssim   \frac{1}{\lambda} \sum_{k=-\infty}^{N_0} 2^{kd/2} \left(\int_{B(0, 2^{k+1})\setminus B(0, 2^k)} |\widehat{I_{\alpha}\mu}(\xi)|^2 \, d\xi \right)^{1/2} \\
			&\lesssim   \frac{1}{\lambda} \sum_{k=-\infty}^{N_0} 2^{kd/2} 2^{-k\frac{(2\alpha-(d-\beta))}{2}} \|\mu\|_{M_b}^{\frac{1}{2}} \|\mu\|_{\dot{B}^{\beta-d}_{\infty,\infty}}^{\frac{1}{2}}\\
			&\lesssim   \frac{1}{\lambda} 2^{-N_0\frac{(2\alpha-(2d-\beta))}{2}} \|\mu\|_{M_b}^{\frac{1}{2}} \|\mu\|_{\dot{B}^{\beta-d}_{\infty,\infty}}^{\frac{1}{2}},
		\end{align*}
		where the last inequality utilizes the hypothesis that $\alpha< d-\frac{\beta}{2}$.
		
		Combining the estimates for $I_1$ and $I_2$, for $\frac{d-\beta}{2}< \alpha < d-\frac{\beta}{2}$, we find
		\begin{align*}
			I_1+I_2 \lesssim \frac{1}{\lambda} 2^{-N_0\frac{(2\alpha-(2d-\beta))}{2}} \|\mu\|_{M_b}^{\frac{1}{2}} \|\mu\|_{\dot{B}^{\beta-d}_{\infty,\infty}}^{\frac{1}{2}} + \frac{1}{\lambda^2} 2^{-N_0(2\alpha-(d-\beta))}  \|\mu\|_{M_b} \|\mu\|_{\dot{B}^{\beta-d}_{\infty,\infty}}.
		\end{align*}
		In particular, the choice of $N_0$ such that $2^{N_0} \simeq \lambda^{-\frac{2}{2\alpha+\beta}} \|\mu\|_{M_b}^{\frac{1}{2\alpha+\beta}} \|\mu\|_{\dot{B}^{\beta-d}_{\infty,\infty}}^{\frac{1}{2\alpha+\beta}}$ yields the estimate
		\begin{align*}
			\left| \left\{ \xi \in \mathbb{R}^d : |\widehat{I_{\alpha}\mu}(\xi)| > \lambda \right\}\right| \leq C \lambda^{-\frac{2d}{2\alpha+\beta}} \|\mu\|_{M_b}^{\frac{d}{2\alpha+\beta}} \|\mu\|_{\dot{B}^{\beta-d}_{\infty,\infty}}^{\frac{d}{2\alpha+\beta}}. 	
		\end{align*}
		This shows $\widehat{I_{\alpha}\mu} \in L^{\frac{2d}{2\alpha+\beta}, \infty}(\mathbb{R}^d)$ and $\|\widehat{I_{\alpha}\mu}\|_{L^{\frac{2d}{2\alpha+\beta}, \infty}} \leq C \|\mu\|_{M_b}^{\frac{1}{2}} \|\mu\|_{\dot{B}^{\beta-d}_{\infty,\infty}}^{\frac{1}{2}}$, which is the claim of the theorem.
	\end{proof}

	\begin{proof}[Proof of Corollary \ref{finite-perimeter}]
		We first recall from \eqref{charBV_byparts} that $D\chi_E \in \dot{B}^{-1}_{\infty,\infty}(\mathbb{R}^d)$.  An application of Theorem \ref{thm-frac-Y-main} to $\mu= D\chi_{E}$ with $\alpha=1$ and $\beta=d-1$ yields
		\begin{align}\label{pre-esti}
			\|\widehat{I_1(D\chi_{E})}\|_{L^{\frac{2d}{d+1}, \infty}} \leq \, C |D\chi_E|(\mathbb{R}^d)^{1/2}.
		\end{align}
		Next we observe that $D\chi_{E} \in M_b(\mathbb{R}^d)$ implies $\widehat{D \chi_{E}}$ is a continuous function and therefore
		\begin{align}\label{sub-rep-form}
			\left|\widehat{\chi_{E}}(\xi)\right| = & \left|\frac{\xi}{|\xi|^2}\cdot \widehat{D \chi_{E}}(\xi)\right|\\
			\nonumber= & \left|\frac{\xi}{|\xi|}\cdot \frac{1}{|\xi|}\widehat{D \chi_{E}}(\xi)\right|\\
			\nonumber \leq & \left|\widehat{I_1(D\chi_{E})}(\xi)\right|.
		\end{align}
		In particular, \eqref{pre-esti} and \eqref{sub-rep-form} together imply the estimate
		\begin{align*}
			\|\widehat{\chi_{E}}\|_{L^{\frac{2d}{d+1}, \infty}} \leq \, C |D\chi_E|(\mathbb{R}^d)^{1/2},
		\end{align*}
		which completes the proof.
	\end{proof}

	\begin{proof}[Proof of Corollary \ref{Sobolev-function}]
		Let $\chi_{E} \in W^{\eta, p}$ with $0<\eta <1$, $1\leq p <\infty$ such that $\eta p <1$. Then by definition,
		\begin{align}\label{eq-comp-frac}
			\|(-\Delta)^{\eta p/2}\chi_{E}\|_{L^1} = & C_{d,\eta, p} \int \left| \int \frac{\chi_{E}(x)-\chi_{E}(y)}{|x-y|^{d+\eta p}} \, dy \right| dx \\
			\nonumber= & C_{d,\eta, p} \int  \int \frac{|\chi_{E}(x)-\chi_{E}(y)|^p}{|x-y|^{d+\eta p}} \, dy\, dx \\
			\nonumber = & C_{d,\eta, p} \|\chi_{E}\|_{W^{\eta, p}}^p,
		\end{align}
where one uses the fact that $\chi_E$ only takes the values $0$ and $1$.

		In particular, Lemma \ref{lem-higher-fraclaplace-measure} implies $(-\Delta)^{\eta p /2} \chi_E \in M_b \cap \dot{B}^{\beta-d}_{\infty,\infty}$ with $\beta=d-\eta p$.
		
		An application of Theorem \ref{thm-frac-Y-main} to $\mu=(-\Delta)^{\eta p/2} \chi_{E}$ with $\alpha= \eta p$ and $\beta= d-\eta p$ yields the estimate
		\begin{align}\label{preli-esti-frac}
			\|\reallywidehat{I_{\eta p} \left( (-\Delta)^{\eta p/2} \chi_{E} \right)} \|_{L^{\frac{2d}{d+\eta p},\infty}} \leq C \|\chi_{E}\|_{W^{\eta, p}}^{p/2}.
		\end{align}  
		The preceding inequality, in combination with the relation
		\begin{align}\label{reprent-frac}
			\widehat{\chi_{E}}(\xi) = \left(I_{\gamma p} \left( (-\Delta)^{\gamma p/2} \chi_{E}  \right)\right)\widehat{\phantom{X}} (\xi)
		\end{align}
		implies the desired result:
		\begin{align*}
			\|\widehat{\chi_{E}}\|_{L^{\frac{2d}{d+\eta p},\infty}} \leq C'\|\chi_{E}\|_{W^{\eta, p}}^{p/2}.
		\end{align*}
		This completes the proof of the corollary.	
		
		%
		
		%

		%
		%
		
		%
		%
		%
		%
		%

	\end{proof}

	\begin{proof}[Proof of Theorem \ref{sharp_p}] 
		Let $0<\beta \leq d$ and let $\varphi \in \mathcal{S}(\mathbb{R}^d)$ be such that $\operatorname*{supp} \widehat{\varphi} \subset B(0,4)$ and $\widehat{\varphi}\equiv 1$ on $B(0,2)$.  We first obtain a bound for $\sup_{t>0} t^{\frac{d-\beta}{2}} \|p_t\ast \varphi\|_{L^{\infty}}$.  For $t>1$, we have
		\begin{align}\label{large-t}
			t^{\frac{d-\beta}{2}} \|p_t\ast \varphi\|_{L^{\infty}} \leq  t^{\frac{d}{2}} \|p_t\|_{L^{\infty}} \|\varphi\|_{L^{1}} \leq C^{\prime},
		\end{align}
		where $C^{\prime}$ is a positive constant independent of $t$. For $0<t<1$, we have
		\begin{align}\label{small-t}
			t^{\frac{d-\beta}{2}} p_t\ast \varphi(x)=& t^{\frac{d-\beta} {2}} \int_{\mathbb{R}^d} \frac{1}{(4\pi t)^{d/2}} e^{-\frac{|x-y|^2}{4t}}  \varphi(y) \, dy\\
			\nonumber = & t^{\frac{d-\beta}{2}} \sum_{k\in \mathbb{Z}}\int_{B(0,2^{k+1}\sqrt{t}) \setminus B(0,2^{k}\sqrt{t})} \frac{1}{(4\pi t)^{d/2}} e^{-\frac{|x-y|^2}{4t}}  \varphi(y) \, dy\\
			\nonumber  \leq & C t^{\frac{d-\beta}{2}} \sum_{k\in \mathbb{Z}} \frac{1}{(4\pi t)^{d/2}} e^{-2^{2k-2}} (2^k\sqrt{t})^{d}\|\varphi\|_{L^\infty(\mathbb{R}^d)} \\
			\nonumber \leq & C \sum_{k\in \mathbb{Z}} e^{-2^{2k-2}} (2^k)^{d}\|\varphi\|_{L^\infty(\mathbb{R}^d)}
			\leq  C^{\prime \prime},
		\end{align}
		where $C^{\prime\prime}$ is a positive constant independent of $t$.   In particular, \eqref{large-t} and \eqref{small-t} imply that $\|\varphi\|_{\dot{B}_{\infty,\infty}^{\frac{\beta-d}{2}}} \lesssim 1$.
		
		Let $N$ be a large positive number and define a family of scalings and translates of $\varphi$ by
		\begin{align*}
			\varphi_{N,j}(x)= N^{d-\beta} \varphi(N(x-j)),\quad j\in \mathbb{Z}^d.
		\end{align*}
		
		Observe that 
		\begin{align*}
			t^{\frac{d-\beta}{2}} \|p_t \ast \varphi_{N,j}\|_{L^{\infty}} = s^{\frac{d-\beta}{2}} \|p_{s} \ast \varphi\|_{L^{\infty}}
		\end{align*}
		where $s=tN^2$, so that $\varphi_{N,j} \in \dot{B}^{\beta-d}_{\infty,\infty}$ for each $j\in \mathbb{Z}^d$, while
		\begin{align*}
			\int_{\mathbb{R}^d} |\varphi_{N,j}(x)| \, dx = N^{-\beta}\|\varphi\|_{L^1}.
		\end{align*}
		
		We next define 
		\begin{align*}
		\varphi_N= \sum_{i=1}^{\lfloor N^{\beta} \rfloor} \varphi_{N,j_i} ,
		\end{align*}
		where the set $\{j_i\}_{i=1}^{\lfloor N^{\beta}\rfloor}$ is sparse in the sense that for every $i_1,i_2 \in \{1, \ldots , \lfloor N^{\beta}\rfloor \}$, $|j_{i_1}-j_{i_2}| >e^{e^N}$.   It follows that
		\begin{align*}
			& \int_{\mathbb{R}^d} |\varphi_{N}(x)| \, dx \leq C_1 \quad \text{and} \\
			& \quad \sup_{t>0} t^{\frac{d-\beta}{2}}\|p_t \ast \varphi_N\|_{L^{\infty}} \leq C_2\max_{1\leq i \leq \lfloor N^{\beta} \rfloor } \sup_{t>0} t^{\frac{d-\beta}{2}}\|p_t \ast \varphi_{N, j}\|_{L^{\infty}} \leq  C_3
		\end{align*}
		where $C_1, C_2,C_3$ are positive constants independent of $N$.  In particular $\varphi_N \in M_b \cap \dot{B}^{\beta-d}_{\infty,\infty}$ for all $N$ and $\|\varphi_N\|_{M_b} , \|\varphi_N\|_{\dot{B}^{\beta-d}_{\infty,\infty}} \leq C_4$ where $C_4$ is positive constant independent of $N$. 
		
		Let $\{r_i(t): t \in [0,1]\}_{i=1}^{\infty}$ be a collection of Rademacher functions and define the random function
		\begin{align*}
			\Phi_{N}(t, x)= \sum_{i=1}^{\lfloor N^{\beta} \rfloor} r_i(t) \varphi_{N,j_i}(x).
		\end{align*}
		The preceding computations concerning $\varphi_N$ and the fact that $|r_i|\leq 1$ for all $i$ yield also that $\|\Phi_N\|_{M_b}, \|\Phi_N\|_{\dot{B}^{\beta-d}_{\infty,\infty}}\leq C_4$ for $C_4$ as before.
		
		Set $U_N= B(0, 2N)\setminus B(0,N)$. Since  $\widehat{\phi_{N,j_i}}(\xi)= N^{-\beta} \widehat{\phi}(\xi/N) e^{2\pi i\xi\cdot j_i}$, on $U_N$, we have $|\widehat{\varphi_{N,j_i}}| = N^{-\beta}$. Using Minkowski's inequality and Khintchine's inequality, we deduce
		\begin{align}\label{Khinctchine-weak-Lp}
			& \| \|\widehat{I_\alpha\Phi_N}(t, \xi) \|_{L^{p, \infty}(d\xi)} \|_{L^1([0,1], dt)}\\
			\nonumber & \geq \| \|\widehat{I_\alpha\Phi_N}(t, \xi)\|_{L^1([0,1], dt)} \|_{L^{p, \infty}(d\xi)} \\
			\nonumber	& \simeq \left\|  \left( \sum_{i=1}^{\lfloor N^{\beta} \rfloor}  |\xi|^{-2\alpha}|\widehat{\varphi_{N,j_i}}(\xi)|^2\right)^{1/2}  \right\|_{L^{p, \infty}(d\xi)}\\
			\nonumber	& \gtrsim \sup \lambda \left|\left\{ \xi \in U_N : |\xi|^{-\alpha} N^{-\beta/2}>\lambda\right\} \right|  ^{1/p} \\
			\nonumber	& \simeq\sup \lambda \left|\left\{ \xi \in U_N :  N^{-\alpha-\beta/2}>\lambda\right\}\right|^{1/p} \\
			\nonumber	& \simeq N^{-\alpha-\beta/2} \sup \lambda \left|\left\{ \xi \in U_N :  \lambda<1\right\}\right|^{1/p}\\
			\nonumber	&\gtrsim N^{-\alpha-\beta/2} |U_N|^{1/p}\\
			\nonumber	& \simeq N^{d/p-\alpha-\beta/2}.
		\end{align}
		
		In the third line, we used the fact that on the annulus $U_N$, we have $ \sum_{i=1}^{\lfloor N^{\beta} \rfloor} |\widehat{\varphi_{N,j_i}}(\xi)|^2\simeq N^{-\beta}$, which follows from the fact that $|\widehat{\varphi_{N,j_i}}| = N^{-\beta}$ for each $i$.
		
		For any choice of $p<\frac{2d}{2\alpha+\beta}$, the right hand sight of \eqref{Khinctchine-weak-Lp} tends to infinity as $N$ tends to $\infty$.  Thus, from the left hand side of \eqref{Khinctchine-weak-Lp} one can find a sequence ${t_N}$ with $0\leq t_N\leq 1$ such that 
		\begin{align*}
			\lim_{N\rightarrow \infty}  \|\widehat{I_\alpha\Phi_N}(t_N, \cdot) \|_{L^{p, \infty}(d\xi)} = \infty.
		\end{align*}
		This completes the demonstration of the theorem.
	\end{proof}

		\begin{proof}[Proof of Theorem \ref{sharp_q}]

			Let $k\in \mathbb{N}$ with $0<k\leq d-1$ and $\alpha \in \left( \frac{d-k}{2}, d-\frac{k}{2}\right)$.  For $\xi = (\xi',\xi'') \in \mathbb{R}^{k+1} \times \mathbb{R}^{d-(k+1)}$ one has 
			\begin{align}\label{Fourier-trans-measure}
				\widehat{\mu_k}(\xi)= 2\pi |\xi'|^{-\frac{k-1}{2}} \mathcal{J}_{\frac{k-1}{2}}(2\pi |\xi'|)
			\end{align}
			where $\mathcal{J}_{\frac{k-1}{2}}$ is the Bessel function of order $\frac{k-1}{2}$, see e.g. ~ \cite[Appendix B.4]{Grafakos1class}.  Moreover, one notes that the asymptotic behavior of the Bessel function is given by
			\begin{align}\label{assymp-Bessel}
				\mathcal{J}_{\frac{k-1}{2}}(|\xi'|)= \sqrt{\frac{2}{\pi |\xi'|}} \cos \left( |\xi'|-\frac{\pi(k-1)}{4}-\frac{\pi}{4}\right) + R_{\frac{k-1}{2}}(|\xi'|)
			\end{align}
			where $R_{\frac{k-1}{2}}$ satisfies $|R_{\frac{k-1}{2}}(|\xi'|)|\leq C |\xi'|^{-3/2}$ for $|\xi'|\geq 1$, see \cite[ Appendix B.8]{Grafakos1class}.  Finally, it is useful for us to observe that for $|\xi'| \geq 10C=:C'$ one has $|R_{\frac{k-1}{2}}(|\xi'|)| \leq \frac{1}{|\xi'|^{1/2}} \times \frac{C}{|\xi'|}< \frac{1}{10}\frac{1}{|\xi'|^{1/2}}$.
			
We are now prepared to establish the claimed lower bound.   Let $r\geq 1$ and $0<t<\infty$.  Then equations \eqref{Fourier-trans-measure} and \eqref{assymp-Bessel} allow us to estimate	
\begin{equation}\label{est-1}
\begin{aligned}
				 \left|\left\{ \xi : |\widehat{I_{\alpha}\mu_k}(\xi)|  \right. \right. &\left. \left. >t  \right\} \right|^{1/r} \\
				 = & \left| \left\{\xi : (2\pi)^{1-\alpha}  |\xi|^{-\alpha} |\xi'|^{-\frac{k-1}{2}}|\mathcal{J}_{\frac{k-1}{2}}(2\pi |\xi'|)|>t \right\}  \right|^{1/r}\\
				 \geq & \left| \left\{\xi: |\xi'|\geq C' , (2\pi)^{1-\alpha}\sqrt{\frac{2}{\pi}} |\xi'|^{-\alpha-\frac{k}{2}}\times\right.\right.\\
				 & \quad \quad \left.\left. \left(|\cos(2\pi |\xi'|-\frac{\pi(k-1)}{4}-\frac{\pi}{4})|-\frac{C}{|\xi'|}\right)>t \right\} \right|^{1/r}\\
				 \geq & \left| \left\{\xi :  |\xi'|\geq C', |\xi|^{-\alpha-\frac{k}{2}}> 10(2\pi)^{\alpha-1}\sqrt{\frac{\pi}{2}}t \right\} \right.\\
				 &\left.\bigcap \left\{\xi' \in \mathbb{R}^{k+1} :  |\xi'|\geq C', \left(|\cos(2\pi |\xi'|-\frac{\pi(k-1)}{4}-\frac{\pi}{4})|-\frac{C}{|\xi'|}\right)>\frac{1}{10} \right\}  \right|^{1/r}\\
				 \geq & \left| \left\{\xi :  |\xi'|\geq C', |\xi|^{-\alpha-\frac{k}{2}}> A_{\alpha}t \right\} \right.\\
				 &\left.\bigcap \left\{\xi' \in \mathbb{R}^{k+1} :  |\xi'|\geq C', |\cos(2\pi |\xi'|-\frac{\pi(k-1)}{4}-\frac{\pi}{4})|>\frac{1}{5} \right\}  \right|^{1/r}\\
				 \geq & C_{k} \left|\left\{\xi :  |\xi'|\geq C', |\xi|^{-\alpha-\frac{k}{2}}> A_{\alpha}t \right\} \right|^{1/r}
\end{aligned}
\end{equation}
where $A_{\alpha}= 10(2\pi)^{\alpha-1}\sqrt{\frac{\pi}{2}}$. 
In the last inequality, we have used the fact that the volume of the ball with cylinder removed
\begin{equation}\label{ball-minus-cylinder}
\left\{\xi :  |\xi'|\geq C', |\xi|^{-\alpha-\frac{k}{2}}> A_{\alpha}t \right\}
\end{equation}
which intersects the union of cylindrical annuli
\begin{equation*}
\left\{\xi' \in \mathbb{R}^{k+1} :  |\xi'|\geq C', |\cos(2\pi |\xi'|-\frac{\pi(k-1)}{4}-\frac{\pi}{4})|>\frac{1}{5} \right\}
\end{equation*}
is comparable to the volume of \eqref{ball-minus-cylinder} itself.
 
We next establish a lower bound for \eqref{ball-minus-cylinder}, for sufficiently small $t>0$:
			\begin{align}\label{esti-2}
				 \left|\left\{\xi :  |\xi'|\geq C',\right. \right. &\left. \left.  |\xi|^{-\alpha-\frac{k}{2}}  > A_{\alpha}t \right\} \right|^{1/r}\\
				\nonumber = & \left| B\left(0, A_{\alpha}^{-\frac{1}{\alpha+\frac{k}{2}}}t^{-\frac{1}{\alpha+\frac{k}{2}}} \right)\setminus \left\{\xi \in B\left(0, A_{\alpha}^{-\frac{1}{\alpha+\frac{k}{2}}}t^{-\frac{1}{\alpha+\frac{k}{2}}} \right): |\xi'|\leq C' \right\}\right|^{1/r}\\
				\nonumber \geq & \left| B\left(0, A_{\alpha}^{-\frac{1}{\alpha+\frac{k}{2}}}t^{-\frac{1}{\alpha+\frac{k}{2}}} \right)\right|^{1/r}- \left|\left\{\xi \in B\left(0, A_{\alpha}^{-\frac{1}{\alpha+\frac{k}{2}}}t^{-\frac{1}{\alpha+\frac{k}{2}}} \right): |\xi'|\leq C' \right\}\right|^{1/r}\\
				\nonumber \geq &  A_{\alpha}^{-\frac{2d}{2\alpha+k}\cdot \frac{1}{r}}t^{-\frac{2d}{2\alpha+k}\cdot\frac{1}{r}} \left( \omega_{d}^{1/r} - 2^{d/r}(C')^{\frac{k+1}{r}} A_{\alpha}^{\frac{2k+2}{2\alpha+k}\cdot\frac{1}{r}}t^{\frac{2k+2}{2\alpha+k}\cdot\frac{1}{r}} \right),
			\end{align}
			where $\omega_d$	is volume of the unit ball in $\mathbb{R}^d$.
			
			Set $p=\frac{2d}{2\alpha+k}$ and let $q<\infty$.  Using the definition of $L^{p,q}(\mathbb{R}^d)$, we write
			\begin{align*}
				\|\widehat{I_{\alpha}\mu_k}\|_{L^{p, q}}^{q}
				= \int_{0}^{\infty} \left( t \left|\left\{\xi : |\widehat{I_{\alpha}\mu_k}(\xi)|>t \right\} \right|^{1/p} \right)^{q} \frac{dt}{t}.
			\end{align*}
			The estimates \eqref{esti-2} and  \eqref{est-1} with $r=p=\frac{2d}{2\alpha+k}$ give a lower bound for the right hand side of the above the expression:  		\begin{align*}
				\|\widehat{I_{\alpha}\mu_k}\|_{L^{p, q}}^{q}
				\geq & \frac{\omega_{d}^{q/r} A_\alpha^{-q}}{2^q} \int_{0}^{c}   \frac{dt}{t}
				\end{align*}
			for any $c$ such that
\begin{align*}
\left( \omega_{d}^{1/r} - 2^{d/r}(C')^{\frac{k+1}{r}} A_{\alpha}^{\frac{2k+2}{2\alpha+k}\cdot\frac{1}{r}}t^{\frac{2k+2}{2\alpha+k}\cdot\frac{1}{r}} \right) \geq \frac{\omega_{d}^{1/r}}{2}
\end{align*}
for all $0<t\leq c$.  As the integral in the right hand side of the above inequality is infinite for any $0<q<\infty$, we conclude that  
			\begin{align*}
				\|\widehat{I_{\alpha}\mu_k}\|_{L^{p, q}}= \infty 
			\end{align*}
			for $p=\frac{2d}{2\alpha+k}$ and $0<q<\infty$.

		\end{proof}

	\section*{Acknowledgments} 
	R.~Basak is supported by the National Science and Technology Council of Taiwan under research grant numbers 113-2811-M-003-007/113-2811-M-003-039.	
	D. Spector is supported by the National Science and Technology Council of Taiwan under research grant numbers 113-2115-M-003-017-MY3 and the Taiwan Ministry of Education under the Yushan Fellow Program.
	D. Stolyarov is supported by the Russian Science Foundation grant 19-71-30002.

\end{document}